\documentclass[11pt,twoside]{article}

\def\titlename{\huge Dual space and hyperdimension of   compact  hypergroups}

\title{\titlename}

\def\authname{Mahmood Alaghmandan and Massoud Amini}
\author{{\normalsize\sc \authname}}

\usepackage{amsmath}
\usepackage{yfonts}

\usepackage[all]{xy}

\usepackage{xcolor}
\definecolor{blue1}{RGB}{32,78,170}
\definecolor{blue2}{RGB}{93,92,160}
\definecolor{blue3}{RGB}{40,51,202}
\definecolor{blue4}{RGB}{0,0,0}
\definecolor{purple1}{RGB}{128,0,128}

\definecolor{El}{rgb}{.4,.9,1}
\usepackage[pdftex,
            pdfauthor={Mamhood},
            linkcolor=blue1,
            citecolor=blue3,
            colorlinks=true]{hyperref}

\usepackage{titlesec}

\titleformat{\section}
  {\normalfont\fontsize{12}{15}\bfseries}{\thesection}{1em}{}

\usepackage[english]{babel}
\usepackage{blindtext}
\usepackage[scaled=.90]{helvet}
\usepackage{xcolor}
\usepackage{titlesec}
\titleformat{\chapter}[display]
  {\normalfont\sffamily\huge\bfseries\color{blue4}}
  {\chaptertitlename\ \thechapter}{20pt}{\Huge}
\titleformat{\section}
  {\normalfont\sffamily\Large\color{blue4}}
{\thesection}{1em}{}

 \titleformat{\subsection}
  {\normalfont\fontsize{12}{14}\sffamily\color{blue4}}
  {\thesubsection}{1em}{}

\newcommand{\ma}[1]{\textcolor{blue2}{\textsf{#1}}}

            \newcounter{pulse}[section] 

\numberwithin{pulse}{section}
\numberwithin{equation}{section}

\newtheorem{theorem}[pulse]{\bf Theorem}
\newtheorem{proposition}[pulse]{\bf Proposition}
\newtheorem{lemma}[pulse]{\bf Lemma}

\newtheorem{lem}[pulse]{\bf Lemma}
\newtheorem{cor}[pulse]{\bf Corollary}

\newtheorem{dummy-eg}[pulse]{\bf Example}
\newtheorem{dummy-rem}[pulse]{\bf Remark}
\newtheorem{dummy-ques}[qq]{\bf Question}
\newenvironment{eg}{\begin{dummy-eg}\upshape}{\end{dummy-eg}\ignorespacesafterend}
\newenvironment{rem}{\begin{dummy-rem}\upshape}{\end{dummy-rem}\ignorespacesafterend}

\newtheorem{dummy-def}[pulse]{\bf Definition}

\usepackage{amsmath,amssymb}

\newenvironment{dfn}{\begin{dummy-def}\upshape}{\end{dummy-def}\ignorespacesafterend}

\newenvironment{proof}{\noindent{\it Proof.}\/}{\hfill$\Box$\newline\ignorespacesafterend}

\newcommand{\norm}[1]{\| #1 \|}
\newcommand{\conj}{{\operatorname{Conj}}}
\newcommand{\SU}{\operatorname{SU}}
\newcommand{\ignore}[1]{{ }}
\newcommand{\supp}{\operatorname{supp}}
\newcommand{\AM}{\operatorname{AM}}
\newcommand{\tr}{\operatorname{tr}}

\begin{document}

\maketitle

\begin{abstract}
We characterize dual spaces and compute   hyperdimensions of irreducible  representations  for two classes of compact hypergroups namely conjugacy classes of compact groups  and compact  hypergroups constructed by joining compact and finite hypergroups. Also studying the representation theory of finite hypergroups, we highlight some interesting differences and similarities between the representation theories of finite hypergroups and finite groups.  Finally, we compute the Heisenberg inequality for compact hypergroups.
\end{abstract}


\vskip2.0em

Richard Vrem studied representation theory of compact hypergroups \cite{vr}. He showed that, similar to the group case, for every irreducible representation $\pi$ of a compact hypergroup  $H$, $\pi$ is of a finite dimension $d_\pi$. 
Here we use $\widehat{H}$ to denote the maximal set of all irreducible  representations of $H$ which are  pairwise inequivalent. The set $\widehat{H}$  equipped with the discrete topology is called the  \ma{dual space} of $H$.

Vrem showed that coefficient functions on  compact hypergroups satisfy a hypergroup analogue of Peter-Weyl relation which is as follows \cite{vr}. For each pair $\pi,\sigma \in \widehat{H}$ there exists a constant $k_\pi$ such that for every  coefficient functions $\pi_{i,j}$ and $\sigma_{k,l}$,
\begin{equation}\label{eq:orth-of-coeficients}
\int_H \pi_{i,j}(x)\overline{\sigma_{k,l}(x)}dx=\left\{
\begin{array}{l c}
\displaystyle  \frac{1}{k_\pi} & \text{when $i=k$, $j=l$, and $\pi=\sigma$}\\
0 & \text{otherwise}
\end{array}
\right. .
\end{equation}
It is proved that $k_\pi\geq d_\pi$.  We call $k_\pi$, the \ma{hyperdimension} of $\pi$  after \cite{mm}. Recall that for a commutative (compact) hypergroup $H$, every representation $\pi$ is one dimensional.

\vskip0.5em

In this   paper, we study dual spaces and hyperdimensions of irreducible representations for compact hypergroups.
First, in Section~\ref{s:plancherel}, we present some preliminaries and simple computations on (commutative) compact hypergroups. It is interesting that the Plancherel theorem holds for commutative hypergroups. We show that the Plancherel measure on the dual space of hypergroups is nothing but the map that assigns each element of $\widehat{H}$ to its hyperdimension.

Second, in Section~\ref{s:classes-compact-hypergroups}, we characterize dual spaces and find hyperdimensions for  two   classes of compact hypergroups.   For a compact group $G$, the conjugacy classes  form a compact hypergroup. This hypergroup was introduced first by Jewett, in \cite{je}, as one of the prominent examples of compact hypergroups constructed on compact groups. Subsection~\ref{ss:dual-of-conj(G)} is dedicated to this class of commutative compact hypergroups. We also present a proof for the duality relation between the compact hypergroup of conjugacy classes and the discrete hypergroup constructed by irreducible representations of a compact group. The majority of the results in this subsection are known for the more general class of orbit hypergroups (look at \cite{center} and \cite{lasser-hartmann}).
  By joining a compact hypergroup and a finite hypergroup, one may construct a new compact hypergroup. This class of compact hypergroups first was defined and studied in \cite{vr2} where the dual space of commutative case was also studied. In Subsection~\ref{ss:join-hypergroups}, we generalize the   result of \cite{vr2} to (not necessarily commutative) compact hypergroup joins and also compute their hyperdimensions. 

Finite hypergroups have been of interest due to their  many applications in number theory, combinatorics, operator algebras and conformal field theory, \cite{wild}. In Section~\ref{s:finite-hypergroups} we study the representation theory of finite hypergroups. It is interesting that although this theory is very similar to the representation theory of finite groups, many dramatic differences appear in non-group cases.
For an amenable Banach algebra $A$, there is an associated amenability constant $\AM(A)$ (as defined by B. E. Johnson).  Vaguely speaking,  amenability constant  lets us measure amenability of   Banach algebras.
In this section we  study the amenability constant of hypergroup algebras for  finite commutative hypergroups and present a concrete formula to compute it. Interestingly we show that the lower bounds and boundary properties of the amenability constant of the center of the group algebras of finite  groups do not hold for  simple examples of finite commutative hypergroups. This study is a generalization of previous studies in \cite{AzSaSp, ma3, yemon-AM} on  ZL-amenability of finite groups.

We finish the paper with Section~\ref{s:uncertainty} on the uncertainty principle of compact hypergroups. 
The classical (Heisenberg) uncertainty principle states that a function and its Fourier transform cannot both be highly concentrated. In quantum mechanics, this implies that it is impossible to determine a particle's position and momentum simultaneously. We prove that a similar fact holds for   compact  hypergroups.
 We see that in the Heisenberg inequality of compact hypergroups, the hyperdimensions play an important role.

\begin{section}{Preliminaries}\label{s:plancherel}

Let $H$ be a compact hypergroup. We  assume that the Haar measure of $H$, denoted by $\lambda_H$,  is normalized unless otherwise is stated.  An  \ma{ (irreducible) representation} $\pi$ of $H$ is
\begin{itemize}
\item[(1)]{ an (irreducible) $*$-representation  from $M(H)$, the Banach $*$-algebra of bounded Borel measures on $H$, into ${\cal B}({\cal H}_\pi)$ for some Hilbert space ${\cal H}_\pi$,}
\item[(2)]{$\pi(e)=I$,}
\item[(3)]{and for each pair $\xi, \eta \in {\cal H}_\pi$, the coefficient function $\mu \mapsto \langle \pi(\mu) \xi, \eta\rangle$ forms a continuous function on $M(H)^+$ with respect to the weak topology.}
\end{itemize}

It is a consequence of this definition that each representation $\pi$ is norm decreasing. For each  irreducible   representation $\pi$ of $H$ and $x\in H$, $\pi(x)$ is a $d_\pi \times d_\pi$ matrix  and therefore the \ma{(hypergroup) character} $x \mapsto \chi_\pi(x)$ which is the trace of  $\pi(x)$ as well as $ x \mapsto \pi_{i,j}(x)$, the \ma{coefficient function} constructed by the $(i,j)$-th coefficient of the matrix $\pi(x)$ are continuous functions on $H$.

 One may easily  apply  the orthogonality relation  (\ref{eq:orth-of-coeficients}) to get the following relation for characters.
\begin{equation}\label{eq:characters-on-compact-hypergroups}
\int_H \chi_\pi(x)\overline{\chi_\sigma(x)}dx=\left\{
\begin{array}{l c}
\displaystyle \frac{d_\pi}{k_\pi} & \text{when $\pi=\sigma$}\\
0 & \text{otherwise}
\end{array}
\right.
\end{equation}
for $\pi, \sigma \in  \widehat{H}$. Therefore, $\norm{\chi_\pi}_2^2=d_\pi/k_\pi$. 

We will use the following  lemma in Section~\ref{s:finite-hypergroups}. Its proof is a straightforward application of  (\ref{eq:orth-of-coeficients}) and is  similar to the group case, so we omit the proof here. 

\begin{lem}\label{l:convolution-coefficients}
Let $\pi_{i,j}$ and $\sigma_{k,\ell}$ be two coefficient functions for representations $\pi, \sigma \in \widehat{H}$ for  a compact hypergroup $H$. Then
\[
\pi_{i,j} * \sigma_{k,\ell} (x) = \left\{
\begin{array}{l l}
0 & \text{if $\pi \neq \sigma$}\\
\frac{1}{k_\pi}  \pi_{i, \ell}(x) & \text{if $\pi =\sigma$}
\end{array} 
\right.
\]
Consequently,  
\[
\displaystyle  {k_\pi}  \chi_\pi *  {k_\sigma}  \chi_\sigma(x)=\left\{
\begin{array}{l l}
0 & \text{if $\pi \neq \sigma$}\\
  {k_\pi}  \;\chi_\pi(x) & \text{if $\pi =\sigma$}
\end{array} 
\right. .
\]
 \end{lem}

For each $\pi \in \widehat{H}$, define $\widehat{f}(\pi)$ to be the matrix $[ \langle f, \pi_{i,j}\rangle ]_{ij=1}^{d_\pi}$ where $\langle \cdot, \cdot \rangle$ is the inner product of $L^2(H)$. 
For each $f\in L^2(H)$,   applying the Fourier transform, we have
\begin{equation}\label{eq:Fourier-transform}
f=\sum_{\pi\in \widehat{H}} k_\pi \sum_{i,j=1}^{d_\pi}  \widehat{f}( \pi)_{i,j}   \pi_{i,j}
\end{equation}
 and the series converges in $L^2(H)$. Hence,
\begin{equation}\label{eq:norm-2}
\norm{f}_2^2 =  \sum_{\pi \in \widehat{H}} k_\pi \sum_{i,j=1}^{d_\pi}  |\widehat{f}(\pi) _{i,j}|^2  =  \sum_{\pi \in \widehat{H}}   k_\pi \norm{ \widehat{f}(\pi )}_2^2 
\end{equation}
for every $f\in L^2(H)$.

In particular if $H$ is commutative,  because every representation is $1$ dimensional, (\ref{eq:Fourier-transform}) is re-written as
\begin{equation}\label{eq:Fourier-commutative}
f=\sum_{\chi\in \widehat{H}}   k_\chi \widehat{f}(\chi) \chi,
\end{equation}
for $\widehat{f}(\chi)= \langle f, {\chi }\rangle$.
Therefore,  $(\sqrt{k_\chi} \chi)$ forms an orthonormal basis for $L^2(H)$, by the orthogonality relation (\ref{eq:characters-on-compact-hypergroups}). Further, $(k_\chi \chi)_{\chi \in \widehat{H}}$   forms the set of all minimal projections of $L^1(H)$, by Lemma~\ref{l:convolution-coefficients}.
It is known that for every compact commutative hypergroup $H$, there is a measure $\varpi$ on $\widehat{H}$ with respect to that, the restriction of the Fourier transform to $ L^2(H)$ forms an isometry onto $\ell^2(\widehat{H}, \varpi)$ (see \cite[Section~2.2]{bl}). The measure $\varpi$ is called the \ma{Plancherel measure} on $\widehat{H}$.

\begin{proposition}\label{p:Plancherel}
Let $H$ be a commutative compact hypergroup. Then the Plancherel measure $\varpi$ on every $\pi \in \widehat{H}$ is equal to $k_\pi$.
\end{proposition}

\begin{proof}
 By  (\ref{eq:characters-on-compact-hypergroups}),   one easily gets $k_\psi \widehat{\psi} =  \delta_\psi$  for each $  \psi\in \widehat{H}$ where $\delta_\psi$ is the point-mass function on $\psi$ whose value is $1$ at $\psi$ and zero everywhere else.  
Hence, for a fixed $\psi \in \widehat{H}$, we get
\begin{eqnarray*}
\varpi(\psi) &=&   \sum_{\chi\in \widehat{H}} \delta_\psi(\chi) \varpi(\chi) = \sum_{\chi\in \widehat{H}} |\delta_\psi(\chi)|^2 \varpi(\chi)\\
&=&  {k_\psi^2} \sum_{\chi\in \widehat{H}} |\widehat{\psi}(\chi)|^2 \varpi(\chi)
=  {k_\psi^2}  \int_H |{\psi}(x)|^2 d\lambda_H(x) =  {k_\psi^2}  \norm{\psi}_2^2 =  {k_\psi}.
\end{eqnarray*}
Note that the first equation in the second line is based on the definition of the Plancherel measure. 
\end{proof}

\begin{eg}\label{eg:G-as-a-compact-hypergroup}
Let $G$ be a compact group. Obviously, $G$ is a compact hypergroup.  Readily based on the Peter-Weyl orthogonality relation for compact groups, $k_\pi=d_\pi$ for every $\pi \in \widehat{G}$.  Hence, for a commutative compact group $G$,  the Plancherel measure on $\widehat{G}$ is constantly $1$.
\end{eg}

Let $K$ be a compact subhypergroup of a commutative   hypergroup $H$.  Then  $K/H$, the set of all  cosets of $K$ in $H$ equipped with the quotient topology through the mapping $p_K:H\rightarrow K/H$, where $p_K(x)=xK$, forms a commutative hypergroup.     Further, if $H$ is compact, so is $H/K$.
The first part of the following corollary was proved in \cite[Proposition~2.2.46]{bl}  (for not necessarily compact hypergroups).   Here, we  let $\widehat{H}_K$ denote the set of all characters of $H$, say $\chi$, such that $\chi(x)=1$ for all $x\in K$.

\begin{cor}\label{t:^H//K}
Let $K$ be a closed subhypergroup of a compact commutative hypergroup $H$. Then the mapping   $\Phi: \widehat{H}_K \rightarrow \widehat{H/K}$ $(\chi \mapsto \chi \circ p_K$) is a bijection such that $k_{\pi}=k_{\Phi(\chi)}$ for all $\chi\in \widehat{H}_K$.
\end{cor}

\begin{proof}
Here we just show the equality  of the hyperdimenstions. To do so,  by \cite[Theorem~1.5.20]{bl}, we have
\begin{eqnarray*}
\frac{1}{k_{\Phi(\chi)}} &=& \int_{H/K} | \Phi(\chi)(xK)|^2 d\lambda_{H/K}(xK) =  \int_{H} | \chi(x)|^2 d\lambda_H(x) = \frac{1}{k_\chi}
\end{eqnarray*}
for $\chi \in \widehat{H}_K$.
\end{proof}

\end{section}

\section{Two classes of compact hypergroups}\label{s:classes-compact-hypergroups}

\begin{subsection}{Conjugacy classes of a compact group}\label{ss:dual-of-conj(G)}

Let $G$ be a compact group and $\conj(G)$ denote the set of all conjugacy classes of $G$. Here for each $x\in G$, we use $C_x$ to denote the conjugacy class of $x$ that is $\{ yxy^{-1}:  y\in G\}$. For each pair $x, y \in G$, the convolution $*$ defined by
\begin{equation}\label{eq:conv-conj(G)}
\delta_{C_x} * \delta_{C_y} = \int_G \int_G \delta_{C_{sxs^{-1}tyt^{-1}}} ds dt
\end{equation}
forms a hypergroup action on $\conj(G)$ when $\conj(G)$ is equipped with the quotient topology carried through the canonical mapping $x \mapsto C_x$.
A function $f\in C(G)$ is called a \ma{class function} if it is invariant on conjugacy classes of $G$. A class function $f \in C(G)$ can be canonically considered as a continuous function on $\conj(G)$.  The Haar measure of $\conj(G)$, denoted by $\lambda_{\conj(G)}$, is characterized as the measure on $\conj(G)$  for that
\begin{equation}\label{eq:Haar-Conj(G)}
\int_G f(x) d\lambda_G(x) = \int_{\conj(G)} f(C_x) d\lambda_{\conj(G)}(C_x)
\end{equation}
for every class function $f$ on $G$.    Note that $\conj(G)$ is a commutative hypergroup, so for each $\chi \in \widehat{\conj}(G)$, $d_\chi=1$.

This class of hypergroups fall into a larger class of commutative hypergroups, called \ma{orbit hypergroups}. Orbit hypergroups are admitted by  [FIA$\overline{]}^B$ locally compact groups $G$ where $B$ is a relatively compact group of automorphisms of $G$ including all inner automorphism. For a detailed reference on this class of hypergroups look at  \cite[8.1]{je}. In \cite{center}, it was shown that the dual object of these commutative hypergroups can be identified with
the set of all $B$-characters on $G$ defined and studied formerly by Mosak \cite{mosak}. This generalizes the first part of the following theorem.

\begin{theorem}\label{t:characters-Conj(G)}
Let $G$ be a compact group and $\conj(G)$ denotes the hypergroup of conjugacy classes of $G$. Then the mapping $\pi \mapsto d_\pi^{-1} \chi_\pi$ is a bijection from $\widehat{G}$ onto $\widehat{\conj}(G)$. Further, for each $\psi \in \widehat{\conj}(G)$, $k_\psi= d_\pi^2$ for $\psi= d_\pi^{-1} \chi_\pi$.
\end{theorem}

\begin{proof}
\ignore{
First note that for each $\pi \in \widehat{G}$, $d_\pi^{-1} \chi_\pi$ is a class function. Hence, for each $x,y$  in $G$, we get
\begin{eqnarray*}
d_\pi^{-1}  \chi_\pi(\delta_{C_x} * \delta_{C_y})
&=&  \int_{G} \int_{G} d_\pi^{-1} \chi_\pi(sxs^{-1}tyt^{-1}) d\lambda_G(s) d\lambda_G(t)\\
&=& d_\pi^{-1} \sum_{i,j=1}^{d_\pi}  \int_G \int_G \pi_{i,j}(sxs^{-1}) \pi_{j,i}(tyt^{-1}) d\lambda_G(s) d\lambda_G(t)\\
&=& d_\pi^{-1} \sum_{i,j,l,k,m=1}^{d_\pi}  \langle  \pi_{i,l} , {\pi_{j,k}}\rangle \; \langle \pi_{j,m}, {\pi_{i,n} } \rangle \;  \pi_{l,k}(x) \pi_{m,n}(y) \ \ \ (\sharp)\\
&=& d_\pi^{-2} \sum_{i,j=1}^{d_\pi}  \pi_{i,i}(x) \pi_{j,j}(y)  = d_\pi^{-1} \chi_\pi(C_x) \; d_\pi^{-1} \chi_\pi(C_y).
\end{eqnarray*}
Note that $(\sharp)$ is based on orthogonality relation of coefficients of $G$.
Therefore, $d_\pi^{-1} \chi_\pi$ belongs to $\widehat{\conj}(G)$.

Conversely,  let $\psi \in \widehat{\conj}(G)$.
Towards a contradiction assume that $\psi$ is equal to no $d_\pi^{-1} \chi_\pi$ for $\pi \in \widehat{G}$.
Note that $\psi$ is also a class function on $G$, so for every $\pi \in \widehat{G}$, by (\ref{eq:characters-on-compact-hypergroups}), we get
\[
d_\pi^{-1} \int_{G} \psi(x) \overline{\chi_\pi(x)} dx = \int_{\conj(G)} \psi(x) \overline{d_\pi^{-1} \chi_\pi(x)} dx=0
\]
This implies that $\psi =0$, as the linear span of  characters of $G$ is dense in the subspace of $L^2(G)$ of all class functions.
}
As we mentioned before, by \cite{center} and \cite{mosak}, the dual object of $\conj(G)$ is identified with the set of all characters of $G$ which are invariant under the conjugations of all inner automorphisms. But the former set is the set of all characters constructed by irreducible representations of $G$. Hence, there is a bijection from $\widehat{G}$ onto $\widehat{\conj}(G)$  through the mapping  $\pi \mapsto \psi_\pi$ where $\psi_\pi(C_x):= d_\pi^{-1} \chi_\pi(x)$ for every conjugacy class $C_x \in \conj(G)$.

Since $\conj(G)$ is commutative, $d_{\psi }=1$ for all $\psi\in \widehat{\conj}(G)$. Hence, for each $\pi \in \widehat{G}$ by applying (\ref{eq:characters-on-compact-hypergroups}), we have
\[
\frac{1}{k_{\psi_\pi}}=\int_{\conj(G)}  |\psi_\pi(C)|^2 d\lambda_{\conj(G)}(C)= \frac{1}{d_\pi^2} \int_G |\chi_\pi(x)|^2 dx = \frac{1}{d_\pi^2}.
\]
\end{proof}

\begin{eg}\label{eg:Z2-rtimes-T}
Let $\Bbb{T}$ denote the  compact group of $\{ x\in \mathbb{C}: x \overline{x}=1\}$ and $\mathbb{Z}_2=\{1,-1\}$.
Therefore for each $\alpha \in \Bbb{Z}_2$, $x \mapsto x^\alpha$ forms a group automorphism on  $\Bbb{T}$. We define $G$ to be the semidirect product of $\Bbb{T}\rtimes \Bbb{Z}_2$ with respect to this action. One simple computation implies that $\conj(G)$ is decomposed into three classes of elements, namely,  $C_{(1,1)}=\{(1,1)\}$, $C_{(x,-1)}=\{(y,-1): y\in \Bbb{T}\}$, and $C_{(x,1)}=\{(x,1), (\overline{x},1)\}$ for all $x\in \Bbb{T}$. The irreducible representations of ${G}$ are constructed by induction  (see \cite[Theorem~6.42]{fo}). There are  two dimensional representations $\pi_n$ for $n\neq 0$ induced from $\Bbb{T}$ into $G$ and two linear representations $\chi_1$ and $\chi_{-1}$ as extensions of (linear) representations of $\Bbb{Z}_2$. Therefore, $k_{\chi_\pi}=4$ for all non-linear characters $\chi_\pi$ associated to representations $\pi$ of $G$ while $k_{\chi_{\pm 1}}=1$ for  the linear representations $\chi_{\pm  1}$.
\end{eg}

\begin{eg}\label{eg:SU(2)}
Let $\SU(2)$ denote the compact group of $2\times 2$ special unitary matrices.  It is straightforward that each conjugacy class  of $\SU(2)$ except $I$ and $-I$ intersects the maximal tori of $\SU(2)$ twice.  Therefore, one may represent $\conj(\SU(2))$ by $[0,\pi]$ (half of the tori). The representation theory of $\SU(2)$ is very well known, for example look at \cite[Theorem~5.39]{fo}. The dual space $\widehat{\SU}(2)$ is represented by $\{ \pi_n: n = 0, 1,2,\ldots\}$ where for each $n$, $\pi_n$ is of dimension $n+1$. Also the character $\chi_n$, constructed by the trace of $\pi_n$, is computed on $\theta \in [0,\pi]$ by
\[
\chi_n(\theta)= \frac{\sin((n+1)\theta)}{\sin(\theta)}.
\]
Therefore,  $\{\psi_n:=(n+1)^{-1}\chi_n : n=0,1,2,\ldots\}$  forms the representation theory of $\conj(\SU(2))$ as a commutative compact hypergroup, where for each $n$, $k_{\psi_n}=(n+1)^2$.
\end{eg}

\ignore{
\noindent {\bf Conjecture.}
The hypergroup algebra of $\conj(G)$ for a compact group $G$ is isometrically isomorphic to the centre of the group algebra of $G$. The centre of the group algebras have been studies with respect to its amenability properties in \cite{AzSaSp} where it was shown that $L^1(\conj(\Bbb{T}\rtimes \Bbb{Z}_2))$ is amenable while $L^1(\conj(\SU(2)))$ is not (see Examples~\ref{eg:Z2-rtimes-T} and \ref{eg:SU(2)}).  Accordingly, it is conjectured in \cite{AzSaSp} that $L^1(\conj(G))$ is amenable if and only if $\{ k_\pi: \pi \in \widehat{G}\}$ is bounded.
}

A commutative hypergroup $H$ is called a \ma{strong hypergroup} if its dual space, $\widehat{H}$, forms a hypergroup whose Haar measure corresponds to the Plancherel measure.
For a locally compact abelian group, this is always the case, but this is not true necessarily for many known examples of commutative hypergroups including many classes of polynomial hypergroups, see \cite{bl}.

The hypergroup structures on the duals of   (not necessarily compact) orbit hypergroups have been studied  in \cite{lasser-hartmann} where a generalized proof for the  following proposition is presented.  Here, to be self-contained,  we present a proof for the compact case which is shorter and  relies on the theory of compact (hyper)groups.

\begin{proposition}\label{p:duality}
Let $G$ be a compact group. Then  $\conj(G)$ and $\widehat{G}$ both form strong hypergroups and they are dual objects of each other.
\end{proposition}



\begin{proof}
The fusion rule for compact   groups is the key point to define a hypergroup action on the dual space of irreducible representations of compact groups. See \cite[Example~1.1.14]{bl} or \cite[Section~3]{ma} for a brief summary. 
In Theorem~\ref{t:characters-Conj(G)}, we showed that the dual object of $\conj(G)$ is isomorphic to $\widehat{G}$ as two discrete sets. Also we showed that for each $\pi \in \widehat{G}$, the Plancherel measure $\varpi(\pi)=d_\pi^2$. But this matches  exactly with the hypergroup Haar measure defined on $\widehat{G}$ based on its fusion rule. 

To prove that the dual object of $\widehat{G}$ is $\conj(G)$, we need to recall that the hypergroup algebra of $\widehat{G}$ is isometrically isomorphic to a subspace $ZA(G):=ZL^1(G) \cap A(G)$ of the Fourier algebra of $G$   where $ZL^1(G)$ denotes the centre of the group algebra of $G$. (Some properties of $ZA(G)$ have been studied extensively in \cite{ma15}.)  It is proved that the maximal ideal space of $ZA(G)$ is homeomorphic to  $\conj(G)$. See \cite[34.37]{he2} or \cite[Proposition~3.1]{ma15} for a generalized proof.  Therefore the dual object of $\widehat{G}$ is homeomorphic to $\conj(G)$.

To finish the proof, we   show that the Haar measure on $\conj(G)$, (\ref{eq:Haar-Conj(G)}), corresponds to the Plancherel measure of $\widehat{G}$, denoted by $\varpi$. To do so, we use this fact that the extension of the Fourier transform on $L^2(\widehat{G})$ is an isometry onto $L^2(\conj(G), \varpi)$. Also, in the proof of \cite[Theorem~3.7]{ma}, it was shown that $L^2(\widehat{G})$ is isometrically isomorphic to $ZL^2(G)=ZL^1(G) \cap L^2(G)$. One simple averaging argument implies that $ZL^2(G)$ is also isometrically isomorphic to $L^2(\conj(G))$. This finishes the proof.
\end{proof}

\end{subsection}

\begin{subsection}{Compact  hypergroup joins}\label{ss:join-hypergroups}

In this  subsection,  we study compact hypergroups constructed by joining compact hypergroups with finite hypergroups, so called \ma{compact hypergroup joins}.  General hypergroup joins   were defined and studied extensively in  \cite{vr2}.

\begin{dfn}\label{d:join-hypergroups}
Suppose $(K,*_K)$ is a compact hypergroup and $(J,*_J)$ is a discrete hypergroup with $K\cap J=\{e\}$ where $e$ is the identity of the both hypergroups. Let $H=K\cup J$ have the unique topology for which $K$ and $J$ are closed
subspaces of $H$. Let $\lambda_K$ be the normalized Haar measure on $K$ and define
the operation $*$ on $H$ as follows:\\
\begin{itemize}
\item{If $s, t\in K$ then $\delta_t*\delta_s=\delta_t*_K \delta_s$.}
\item{If $s,t\in J$ and $s\neq \tilde{t}$ then $\delta_s*\delta_t=\delta_s*_J \delta_t$.}
\item{If $s\in K$ and $t\in J\setminus\{e\}$ then $\delta_s*\delta_t=\delta_t*\delta_s=\delta_t$.}
\item{If $s\in J$ and $\delta_s *_J \delta_{\tilde{s}} =\sum_{t\in J}\alpha_t \delta_t$,
\[
\delta_s*\delta_{\tilde{s}}=\alpha_e \lambda_K + \sum_{t\in J\setminus \{e\}} \alpha_t\delta_t.
\]}
\end{itemize}
We call $H$ the  \ma{hypergroup join} of $K$ and $J$ and write $H=K \vee J$.
\end{dfn}

 If the discrete hypergroup $J$ is finite, the  hypergroup join $K \vee J$ forms a compact hypergroup.
It should be noted that   $K\vee J$ and $J\vee K$ are not necessarily  equal for two non-equal   finite hypergroups $J$ and $K$.

The following lemma is a generalization of   \cite[Lemma~3.1 ]{vr2}.

\begin{lemma}\label{l:subset-of-^H}
Suppose $H=K\vee J$ is a compact hypergroup. Then each $\pi\in \widehat{K}\setminus \{1\}$ extends to a representation $\Phi(\pi)\in \widehat{H}$ via
\[
\Phi(\pi)(x)=\left\{
\begin{array}{l l}
\pi(x) & x\in K\\
0 & x\in J\setminus \{e\}
\end{array}
\right.\]
Also for each $\pi\in \widehat{J}$ there is some $\Phi(\pi)\in \widehat{H}$ such that
\[
\Phi(\pi)(x)=\left\{
\begin{array}{l l}
\pi(x) & x\in J\\
I_\pi & x\in K
\end{array}
\right.\]
\end{lemma}

\begin{proof}
First let us consider the case that $\pi\in \widehat{K}\setminus \{1\}$. Clearly, for each $s,t\in K$, $\Phi(\pi)(\delta_x*\delta_y)=\Phi(\pi)(x)\Phi(\pi)(y)$, since $\Phi(\pi)|_K=\pi$.  If $x,y \in J$ and $x\neq \tilde{y}$ then $\Phi(\pi) (\delta_s*\delta_t)=\Phi(\pi)(\delta_s*_J \delta_t)=0=\Phi(\pi)(s)\Phi(\pi)(t)$.
If $s\in K$ and $t\in J\setminus\{e\}$ then $\Phi(\pi)(\delta_s*\delta_t)=\Phi(\pi)(t)=0$. If $s\in J$ and $\delta_s *_J \delta_{\tilde{s}} =\sum_{t\in J}\alpha_t \delta_t$,
\[
\Phi(\pi) (\delta_s*\delta_{\tilde{s}})=\alpha_e \pi(\lambda_K) + \sum_{t\in J\setminus \{e\}} \alpha_t \pi(t).
\]
Here, note that for each $1\neq \pi \in \widehat{K}$, $\pi(\lambda_K)=0$ by (\ref{eq:orth-of-coeficients}).
Hence, $\Phi(\pi)(\delta_s*\delta_{\tilde{s}})=0=\Phi(\pi)(s)\Phi(\pi)(\tilde{s})$.
Moreover, clearly $\Phi(\pi)$ is an $*$-continuous irreducible representation as is $\pi$.

Now let $\pi\in \widehat{J}$.  If $s, t\in K$ then $\Phi(\pi)(\delta_t*\delta_s)=I_\pi=\Phi(\pi)(t)\Phi(\pi)(s)$. If $s,t\in J$ and $s\neq \tilde{t}$ then $\Phi(\pi)(\delta_s*\delta_t)=\Phi(\pi)(s) \Phi(\pi)(t)$ as $\Phi(\pi)|_J=\pi$. If $s\in K$ and $t\in J\setminus\{e\}$ then $\Phi(\pi)(\delta_s*\delta_t)=\pi(t)=\Phi(\pi)(s)\Phi(\pi)(t)$. And eventually, if $s\in J$ and $\delta_s *_J \delta_{\tilde{s}} =\sum_{t\in J}\alpha_t \delta_t$,
\[
\Phi(\pi)(\delta_s*\delta_{\tilde{s}})=\alpha_e \Phi(\pi)(\lambda_K) + \sum_{t\in J\setminus \{e\}} \alpha_t\pi(t).
\]
Note that $\Phi(\pi)(\lambda_K)=I_\pi=\pi(e)$ since $\lambda_K$ is the normalized Haar measure on $K$; therefore, $\Phi(\pi)(\delta_s*\delta_{\tilde{s}})= \pi(\delta_s*\delta_{\tilde{s}})=\pi(s)\pi(\tilde{s})=\Phi(\pi)(s)\Phi(\pi)(\tilde{s})$.
Similarly, $\Phi(\pi)$ is an $*$-continuous irreducible representation as is $\pi$.
\end{proof}

\begin{rem}
The proof of Lemma~\ref{l:subset-of-^H} implies that \cite[Lemma~3.1]{vr2} cannot be accurate since Vrem has not excluded the trivial representation of $\widehat{H}$. One may note that in the proof of Lemma~3.1, he assumed that $\int_L \chi(x)dx=0$ which is not precise regarding  the trivial character $\chi\equiv 1$.  Consequently, \cite[Theorem~3.2]{vr2} should be slightly modified correspondingly. Ironically, Vrem has considered the redundant of the identity for the dual case in Theorem~3.3 in \cite{vr2}.
\end{rem}

 For the rest of this subsection,  let us assume that the Haar measure of $H$, $\lambda_H$, is  normalized  and the Haar measure of $J$,   $\lambda_J$, is so that   $\lambda_J(e)=1$. Recall that since $J$ is finite $\lambda_J(J) <\infty$.  The following theorem is the main result of this subsection.

\begin{theorem}\label{t:^H=^K-&^J}
Let $H=K\vee J$ be a compact hypergroup and $J\neq \{e\}$. Then there is a bijection $\Phi$  from $ (\widehat{K}\setminus\{1\} ) \cup\widehat{J}$ onto $\widehat{H}$. Moreover, for each $\pi \in  (\widehat{K}\setminus\{1\}) \cup\widehat{J}$, $d_{\Phi(\pi)}=d_\pi$ and
\[
k_{\Phi(\pi)}= \left\{
\begin{array}{l l}
k_\pi \lambda_J(J) & \pi \in \widehat{K}\setminus \{1\}\\
k_\pi & \pi \in \widehat{J}
\end{array}
\right. .
\]
\end{theorem}

\begin{proof}
In Lemma~\ref{l:subset-of-^H}, we showed that $\Phi$ is an injective mapping into $\widehat{H}$.
Let $\sigma \in \widehat{H}$, we find some $\pi \in (\widehat{K}\setminus\{1\}) \cup\widehat{J}$ such that $\Phi(\pi)=\sigma$.
 By \cite[Theorem~2.3]{vr2}, there is a subset $P \subseteq \widehat{H}$ such that $ \rho|_{J \setminus \{e\}} =0$ for all $\rho \in \widehat{H}\setminus P$ while $  \rho|_K =I_\pi$  for all $\rho \in P$.  If $\sigma$ belongs to $P$, we need to show $\pi := \sigma|_J$ is an irreducible  representation of $(J,*_J)$. Note that $\pi$ is a hypergroup homomorphism, since $\pi(e)=\sigma(\lambda_K)=I$. Moreover, $\sigma|_K=I$ guarantees that $\pi$ is irreducible if and only if $\sigma$ is irreducible. Therefore, $\pi\in \widehat{J}$ and $\sigma=\Phi(\pi)$.

If $\sigma \in \widehat{H}\setminus P$, we show that $\pi:=\sigma|_K$ belongs to  $\widehat{K}\setminus\{1\}$. To do so, first note that, $\pi$ is clearly a homomorphism  with respect to $*_K$ and it is also irreducible.  Further, since $\sigma|_{J\setminus\{e\}}=0$ and the topology on $K$ is corresponding to the topology inherited from $H$ into $K$, $\pi$ is continuous with respect to the topology of $K$.
We should show that $\pi$ cannot be the trivial representation $1$ on $K$. Towards a contradiction assume that $\pi\equiv 1$. If $J\neq \{e\}$, there is some $s\neq e$ such that $\delta_s *_J \delta_{\tilde{s}} =\sum_{t\in J}\alpha_t \delta_t$, therefore,
\[
0=\pi'(s)\pi'(\tilde{s})=\pi'(\delta_s*\delta_{\tilde{s}})=\alpha_e \pi'(\lambda_K) + \sum_{t\in J\setminus \{e\}} \alpha_t \pi'(t)= \alpha_e \pi(\lambda_K)=\alpha_e\neq 0.
\]
Therefore, $\pi\neq 1$.

The fact that $d_\pi=d_{\Phi(\pi)}$ is immediate based on the first part of the proof. To study hyperdimensions, we need to apply the decomposition of the Haar measure of $H$ obtained in \cite{vr2}, that is, $\lambda_H = \lambda_K + \lambda_J'$ where $\lambda_K$ is the normalized Haar measure of $K$ and $\lambda'_J(x)=\lambda_J(x)$ for every $x\in J\setminus\{e\}$ and $\lambda_J'(x)=0$ otherwise. Furthermore,
\[
\norm{\chi_{\Phi(\pi)}}_2^2=\frac{d_{\Phi(\pi)}}{k_{\Phi(\pi)}}\lambda_H(H)= \frac{d_\pi}{k_{\Phi(\pi)}} \lambda_J(J),
\] as an immediate consequence of (\ref{eq:characters-on-compact-hypergroups}). Hence,
\begin{equation}\label{eq:k-of-joints}
\frac{d_\pi}{k_{\Phi(\pi)}} \lambda_J(J)= \int_H |\chi_{\Phi(\pi)}(x)|^2 d\lambda(x)=\int_{K} |\chi_{\Phi(\pi)}(x)|^2 d\lambda_K(x) + \sum_{x\in J\setminus\{e\}} |\chi_{\Phi(\pi)}(x)|^2 \lambda_J(x).
\end{equation}

First let $\pi\in \widehat{K} \setminus\{1\}$. Therefore, $\chi_{\Phi(\pi)}(x)=0$ for all $x\in J\setminus\{e\}$; hence,
\[
\lambda_J(J)\frac{d_{\pi}}{k_{\Phi(\pi)}}= \int_K |\chi_{\pi}(x)|^2 d\lambda_K(x)=\frac{d_\pi}{k_\pi}.
\]
Thus $k_{\Phi(\pi)}$ has to be  $k_\pi \lambda_J(J)$.

Second, let $ \pi \in \widehat{J}$. Therefore, $\chi_{\Phi(\pi)}|_K\equiv d_\pi$. Hence,
\[
\lambda_J(J)\frac{d_\pi}{k_{\Phi(\pi)}}= d_\pi^2  + \sum_{x\in J\setminus\{e\}} |\chi_{\pi}(x)|^2 \lambda_J(x)= \sum_{x\in J} |\chi_{\pi}(x)|^2 \lambda_J(x)=\lambda_J(J)\frac{d_\pi}{k_\pi}.
\]
And this implies that $k_{\Phi(\pi)}=k_\pi$.
\end{proof}

In the following example, using compact hypergroup joins, for each positive integer $n$, we construct a (commutative) hypergroup  that has representations of hyperdimension $p$ for a given $1<p<\infty$.

\begin{eg}\label{eg:each-integer}
Let $\Bbb{T}$ be the torus as a compact group and  $H_p$ the hypergroup defined in Example~\ref{eg:hyperdimension-p} for  $1<p< \infty$. Then the compact hypergroup  $\Bbb{T} \vee H_p$ is an (infinite) compact hypergroup join whose representations are either of hyperdimension $p$ or $p+1$.  
\end{eg}

\end{subsection}
 
\vskip1.0em

\begin{section}{Finite hypergroups}\label{s:finite-hypergroups}

Let $H$ be a discrete hypergroup. For each $x\in H$, it is known that function $\lambda$ defined by $\lambda(x)=(\delta_{\tilde{x}} * \delta_{{x}})(e)^{-1}$ forms a Haar measure on $H$. Therefore, $\lambda(x)\geq 1$ and the equality holds if and only if $x$ is invertible in $H$. Thus, if $\lambda(x)=1$ for every $x\in H$, $H$ is a group. 
In this section $H$  is  a finite hypergroup and $\lambda_H$ (or simply $\lambda$ if there is no risk of confusion) is the aforementioned  Haar measure on $H$. We use $\ell^1(H, \lambda)$ to denote the hypergroup algebra.  Note that $\lambda$ is not normalized unless $H$ is a trivial hypergroup. So we adjust some constants when we use results proved in the previous sections as there we   assumed that the Haar measure is normalized.

\vskip1.0em

Here by $|A|$ we mean the (finite) cardinal of a set $A$. 

\begin{proposition}\label{p:order-finite-hypergroups}
Let $H$ be a finite hypergroup. 
\begin{itemize}
\item[(1)]{Then  $\sum_{\pi \in \widehat{H}} d_\pi^2=|H|$. In particular if $H$ is commutative, $|H|=|\widehat{H}|$.}
\item[(2)]{Then $\sum_{\pi \in \widehat{H}}  k_\pi d_\pi = \lambda(H)$. In particular if $H$ is commutative, $\lambda(H)=\varpi(\widehat{H})$.}
\end{itemize}
\end{proposition}

\begin{proof}
The proof of (1) is simply based on this fact that $\ell^2(H, \lambda)$ is a finite dimensional Hilbert space with $\{ \pi_{i,j}: \ \pi \in \widehat{H}, i,j =1, \ldots, d_\pi\}$ and $\{ \delta_x: x\in H\}$ as two of its basis. 

To prove (2), note that for $f=\delta_e \in \ell^2(H,\lambda)$, by (\ref{eq:norm-2}) and adjusting the normalization,
\[
 \norm{f}_2^2 = \frac{1}{\lambda(H)} \sum_{\pi \in \widehat{H}} k_\pi \norm{ \widehat{f}(\pi)}_2^2.
\]
But,  on one hand $\norm{f}_2^2= 1$ and on the other hand, $\widehat{f}(\pi)=I_{d_\pi}$.
\end{proof}

\begin{cor}\label{c:when-group-finite}
Let $H$ be a finite   hypergroup. Then $H$ is a group if and only if $k_\pi=d_\pi$ for all $\pi \in \widehat{H}$.
\end{cor}

\begin{proof}
If $H$ is a group, it is  known that $k_\pi=d_\pi$ for every $\pi \in \widehat{H}$, by Example~\ref{eg:G-as-a-compact-hypergroup}. Conversely, let $k_\pi=d_\pi$ for every $\pi \in \widehat{H}$. Therefore, by Proposition~\ref{p:order-finite-hypergroups} we have
\[
|H| = \sum_{\pi \in \widehat{H}} k_\pi d_\pi = \lambda(H).
\]
And this implies that for every $x\in H$, $\lambda(x)=1$ which imposes $H$ to become a group.
\end{proof}

\begin{proposition}\label{p:orthogonality-for-finite-hypergroups}
Let $H$ be a finite commutative hypergroup. Then the following orthogonality relations hold on $H$.
\begin{equation}
\sum_{ x \in H} \chi_\pi(x) \overline{\chi_\sigma(x)} \lambda(x) = \left\{
\begin{array}{l l}
0 & \text{if $\pi \neq \sigma$},\\
\frac{ \lambda(H)}{k_\pi}& \text{if $\pi =\sigma$}.
\end{array}
\right.
\end{equation}
and
\begin{equation}\label{eq:orthogonality-dual-finite}
\sum_{\pi \in \widehat{H}} \chi_\pi(x) \overline{\chi_\pi(y)} k_\pi = \left\{
\begin{array}{l l}
0 & \text{if $x\neq y$}\\
\frac{\lambda(H)}{\lambda(x)} & \text{if $x=y$}
\end{array}\right. .
\end{equation}
\end{proposition}

\begin{proof}
Note that   the first orthogonality relation   is nothing but (\ref{eq:characters-on-compact-hypergroups}). To prove the second one, let $A$ be a $|H| \times |H|$ matrix whose rows are labelled by elements of   $\widehat{H}$ and whose columns are labelled by elements of $H$ with entries  $a_{x, \pi} = \chi_\pi(x) \sqrt{\lambda(x) k_\pi}/ \sqrt{\lambda(H)}$. Thus by the first part of this proposition, the rows of the matrix $A$ are orthonormal. This means that $A$ is unitary and hence its columns are also orthonormal, which finishes the proof.
\end{proof} 
 
 \begin{rem}\label{r:dual-finite-no-hypergroup}
 Note that in Proposition~\ref{p:order-finite-hypergroups} and especially in Proposition~\ref{p:orthogonality-for-finite-hypergroups} we did not assume that $\widehat{H}$ is a hypergroup with respect to the Plancherel measure. 
 \end{rem}
 
 \begin{eg}\label{eg:finite-hypergroup-without-dual}
 Let $H=\{e,a,b\}$ the hypergroup presented in \cite[Example~9.1C]{je}. It is shown that $\widehat{H}$ does not form a hypergroup for $\widehat{H}=\{1, \chi, \psi\}$. The convolution formulas computed  in \cite{je} imply that the Haar measure $\lambda$ on $H$ takes values $1, 4, 4$ for $e,a, b$ respectively. One also may compute the hyperdimensions based on their presence in (\ref{eq:orth-of-coeficients}) and gets $k_1=1$, $k_\chi= 36/17$, and $k_\psi=100/17$. Note that regarding these hyperdimensions and the character table of $H$, the orthogonality relation (\ref{eq:orthogonality-dual-finite}) holds, although $\widehat{H}$ is not a hypergroup.
 \end{eg}
 
As we observed in Example~\ref{eg:finite-hypergroup-without-dual}  unlike dimensions, the hyperdimensions of a compact hypergroup are not necessarily positive integers. Even more, in the following example we observe that for every real number $1\leq p < \infty$, there is a (commutative) hypergroup $H_p$ of order $2$ which has $p$ as a hyperdimension.

\begin{eg}\label{eg:hyperdimension-p}
Let $1<p<\infty$ be a fixed real number.  Define $H_p:=\{0,a\}$ where $\delta_a * \delta_a = (1/p) \delta_0 + (1-1/p) \delta_a$.  Note that this implies that $\lambda(a)=p$.
One can easily observe that the dual object of $H$ is nothing but $\widehat{H}=\{1, \chi\}$ where $\chi(a)=-1/p$.  But this implies that $\norm{\chi}_2^2= (p+1)/p$ and therefore, $k_\chi=p$.
\end{eg}

\vskip2.0em

If $A$ is a Banach algebra, we let $A  {\otimes}_\gamma A$ denote the projective tensor product of $A$ with itself.  We say $A$ is \ma{amenable} if it admits a \ma{bounded approximate diagonal (b.a.d.)} that is  a bounded net $(m_\alpha ) \subseteq  A \otimes_\gamma A$ which satisfies
\[
{\bf m}(m_\alpha) a \rightarrow a, \ \ \ a \; {\bf m}(m_\alpha)\rightarrow a, \ \ \ \ \text{ and} \ \ \ a\cdot m_\alpha - m_\alpha  \cdot a\rightarrow 0
\]
for $a$ in $A$, where ${\bf m} : A \otimes_\gamma A \rightarrow A$ is the multiplication map, and the module actions of $A$ on $A \otimes_\gamma  A$ are given on elementary tensors by $a\cdot  ( b \otimes c ) = ( a b ) \otimes c$ and $(b \otimes c ) \cdot a = b \otimes( c a)$.  This is not the original definition of amenability but it is equivalent to the cohomological one.

 Note that if $A$ is a finite dimensional commutative amenable Banach algebra,   there is a unique (\cite{nico-mahya}) element $\Delta \in A\otimes_\gamma A$ so that ${\bf m}(\Delta)=e_A$ and $a \cdot \Delta =\Delta\cdot a$ for every $a\in A$ and the identity $e_A$. $\Delta$ is called the \ma{diagonal} of $A$.  
We can quantify amenability via the \ma{amenability constant}, which was defined in \cite{jo}. Let 
\[
\AM(A) = \inf \{ \sup_\alpha \norm{m_\alpha}:\  (m_\alpha) \ \ \text{is a b.a.d. for $A$}\}
\]
where we allow the infimum of an empty set to be $\infty$. Again for a finite dimensional amenable commutative Banach algebra $A$, $\AM(A)=\norm{\Delta}$.

For a locally compact group it is known that the group algebra is amenable if and only if its amenability constant is $1$ (see \cite[Corollary 1.11]{sto}). 
For a finite group $G$, the amenability constant of the center of the group algebra, denoted by $Z\ell^1(G)$ has been studied before in \cite{AzSaSp, ma3, yemon-AM}.  Note that $Z\ell^1(G)$ is nothing but the hypergroup algebra of $\conj(G)$.   In the following we generalize these studies by  computing the amenability constant for   finite commutative  hypergroups and observe that how different the results could be in comparison to the ones for $Z\ell^1(G)$.

The following theorem and its proof are a hypergroup adaptation of \cite[Theorem~1.8]{AzSaSp}. 

\begin{theorem}\label{t;AM-finite-hypergroup}
Let $H$ be a finite commutative hypergroup with the Haar measure $\lambda$. Then 
\[
\AM(\ell^1(H,\lambda)) = \frac{1}{\lambda(H)^2} \sum_{x,y \in H} \left| \sum_{\chi \in \widehat{H}}   {k_\chi^2}  \chi (x) \overline{\chi (y)}\right| \lambda(x) \lambda(y).
\]
Also   $\AM(\ell^1(H,\lambda))\geq 1$ and the equality  $\AM(\ell^1(H,\lambda))=1$ holds if and only if $H$ is a   group. 
\end{theorem}

\begin{proof}
Using Lemma~\ref{l:convolution-coefficients} one can check that for
\begin{equation}
\displaystyle \Delta =\frac{1}{\lambda(H)^2}   \sum_{\chi \in \widehat{H}} {k_\chi^2}   \chi  \otimes  \chi 
\end{equation}
we have $\psi  \cdot \Delta = \Delta \cdot \psi $ and even more, $\psi * {\bf m}(\Delta)= {\bf m}(\Delta)* \psi = \psi$ for every character $\psi \in \widehat{H}$.  But note that the set of characters  is a basis for $\ell^1(H,\lambda)$. Thus $\Delta$ is the unique diagonal of $\ell^1(H, \lambda)$. 

 So to compute the amenability constant of $\ell^1(H, \lambda)$ it is enough to compute the $1$-norm of $\Delta$ as follows.
\begin{eqnarray*}
\AM (\ell^1(H,\lambda))  &=&  \frac{1}{\lambda(H)^2}  \sum_{x,y \in H}  \left|\sum_{\chi \in \widehat{H}}  {k_\chi^2}   \chi(x)  \overline{\chi(y)} \right| \lambda(x) \lambda(y) \\
&\geq&    \frac{1}{\lambda(H)^2}  \sum_{x\in H}   \sum_{\chi \in \widehat{H}} {k_\chi^2}    |\chi (x)|^2  \lambda(x)^2 \\
&\geq &    \frac{1}{\lambda(H)^2}    \sum_{\chi \in \widehat{H}}  {k_\chi^2}   \sum_{x\in H} |\chi (x)|^2  \lambda(x)  \ \ \ \ \ \ \ \ \ \ \ \ \ \ \ \ \ \ \ \ \ \ \ \ (\dagger)\\
&\geq &  \frac{1}{\lambda(H)^2}     \sum_{\chi \in \widehat{H}}   {k_\chi^2}   \frac{\lambda(H)}{k_\chi} \\
&\geq&  \frac{1}{\lambda(H)}  \sum_{\chi \in \widehat{H}} k_\chi =1.  
\end{eqnarray*}

It is known  that for an amenable locally compact  group $H$, the amenability constant of the group algebra is $1$ (see \cite[Corollary 1.11]{sto}). Conversely, if $H$ is not a group there should be at least one $x\in H$ so that $\lambda(x)>1$.  Meanwhile there is some character $\chi $ so that $\chi (x) \neq 0$. Therefore, the inequality $(\dagger)$ has to be strict. Hence,  $\AM(\ell^1(H,\lambda))>1$.
\end{proof}

\begin{eg}\label{eg:AM(Hp)}
Let $H_p$ be the commutative hypergroup introduced in Example~\ref{eg:hyperdimension-p} for $1<p<\infty$. By Theorem~\ref{t;AM-finite-hypergroup}, we have
\[
AM(\ell^1(H_p, h))=\frac{5p^2 - 2p +1}{(p+1)^2}.
\]
Note that $p \mapsto \AM(\ell^1(H_p,\lambda))$ is an increasing function whose range is the interval $(1,5)$. 
\end{eg}

\begin{rem}\label{r:norm-of-idempotents}
Note that for a commutative finite hypergroup $H$, the diagonal $\Delta \in \ell^1(H \times H, \lambda \times \lambda)$ is an idempotent. Example~\ref{eg:AM(Hp)} implies that for every $r>1$, we may find a commutative hypergroup which has an idempotent whose $1$-norm  is $r$. Compare this   observation with  Saeki's result, in \cite{saeki}, that says for an abelian locally compact group $G$ and for every non-contractive idempotent $\mu \in M(G)$, $\norm{\mu}_{M(G)} \geq  (1+ \sqrt{2})/2$. 
\end{rem}

\begin{eg}\label{eg:AM-conj(G)}
Let $G$ be a finite group. For the finite hypergroup $\conj(G)$ with the Haar measure $\lambda_\conj(C)=|C|$, the formula of $\AM(\ell^1(\conj(G), \lambda_\conj))$ corresponds to the one in \cite[Theorem~1.8]{AzSaSp} computed for $Z\ell^1(G)$. 
Also for the finite hypergroup $\widehat{G}$ with the Haar measure $\lambda_{\widehat{G}}(\pi)=d_\pi^2$, the formula obtained for $\AM(\ell^1(\widehat{G}, \lambda_{\widehat{G}}))$ corresponds to the one in \cite[Proposition~4.2]{ma15} computed for $ZA(G)$.
\end{eg}

\end{section}

\begin{section}{Uncertainty principle for compact hypergroups}\label{s:uncertainty}

The uncertainty principle has been  studied   in special settings such as $\Bbb{R}^n$ as well as in more general settings such as locally compact groups and in particular for compact groups, and a variety of results concerning lower bounds for the product of the measures of the support of a function and the support of its Fourier transform have been derived.   

The uncertainty principle on commutative hypergroups has been studied before. Many researchers considered different variations of the uncertainty  inequality for a variety of commutative hypergroups. To name a few,   {S}turm-{L}iouville hypergroups, \cite{uc1},  commutative hypergroups with $1$ not in the support of the Plancherel measure, \cite{uc3}, finite and $\sigma$-compact hypergroups, \cite{uc4}, and ultraspherical expansions, \cite{uc5}, were studied for this property.
  In this section we focus on (not necessarily commutative) compact hypergroups and prove a Heisenberg inequality for them. 
   
The main observation for the proof of the  following theorem is inspired from \cite{chua}.  In the following $\tr(A)$ denotes the trace of a matrix $A$.

\begin{theorem}\label{t:uncertainty-1}
Let $H$ be a compact hypergroup with the Haar measure $\lambda$. Then for each $f \in L^2(H)$, 
\[
\lambda(H) \leq   \lambda(\supp(f)) \sum_{\pi \in \supp(\widehat{f})}  k_\pi d_\pi.
\]
\end{theorem}

\begin{proof}
Without loss of generality assume that $\lambda(H)=1$. 
Let $f \in L^2(H)$. If $\supp(\widehat{f})$ is infinite there is nothing left to be proved. So assume that $\supp(\widehat{f})$ is finite. In this case, $f$ is continuous and therefore for an arbitrary  $x\in H$, applying (\ref{eq:Fourier-transform}) we get
\begin{eqnarray*}
|f(x)| &\leq&  \sum_{\pi \in \supp(\widehat{f})} k_\pi |\widehat{f}(\pi)_{i,j} \pi_{i,j}(x)|\\
&=& \sum_{\pi \in \supp(\widehat{f})} k_\pi |\widehat{f}(\pi)_{i,j} \overline{\pi}_{j,i}(\tilde{x})|\\
&=& \sum_{\pi \in \supp(\widehat{f})} k_\pi |\tr( \widehat{f}(\pi) \overline{\pi}(\tilde{x}))|.
\end{eqnarray*}
Note that by properties of the Hilbert-Schmidt norm on matrices and since $\overline{\pi}$ is a contractive representation on $M(H)$,
\[
|\tr( \widehat{f}(\pi) \overline{\pi}(\tilde{x}) )|  \leq \norm{\widehat{f}(\pi)}_2  \norm{\overline{\pi}(\tilde{x})}_2  \leq \norm{\widehat{f}(\pi)}_2  \sqrt{d_\pi} \norm{\overline{\pi}(\tilde{x})}  \leq \sqrt{d_\pi} \norm{\widehat{f}(\pi)}_2.
\]
Therefore, by  H\"{o}lder's inequality, one gets
\begin{eqnarray*}
|f(x)|^2 &\leq&   \left( \sum_{\pi \in \supp(\widehat{f})} k_\pi \norm{\widehat{f}(\pi)}_2^2\right) \left(\sum_{\pi \in \supp(\widehat{f})}   k_\pi d_\pi \right) \\
&=&  \norm{f}_2^2 \sum_{\pi \in \supp(\widehat{f})}  k_\pi d_\pi \\
&\leq& \norm{f}_\infty^2\;  \lambda(\supp(f)) \; \sum_{\pi \in \supp(\widehat{f})}   k_\pi d_\pi.
\end{eqnarray*}
\end{proof}

\begin{eg}\label{eg:uncertainty-compact-groups}
For a compact group $G$, for a function $f\in L^2(G)$, Theorem~\ref{t:uncertainty-1} implies the classical Heisenberg inequality
\begin{equation}\label{eq:uncertainty-compact}
\lambda(G) \leq \lambda_G(\supp(f)) \sum_{\pi \in \supp(\widehat{f})} d_\pi^2.
\end{equation}
If $f$ is a central function, i.e. $f\in ZL^2(G)$. Then $f$ can be considered as a function in $L^2(\conj(G))$. Note that in this case, for each $\pi \in \widehat{G}$, $\psi:=d_\pi^{-1} \chi_\pi  \in \widehat{\conj}(G)$ is linear while $k_{\psi }=d_\pi^2$. Therefore, Theorem~\ref{t:uncertainty-1} still implies the same inequality (\ref{eq:uncertainty-compact}). 

If $G$ is a finite group, then for each $f\in Z\ell^2(G)(=\ell^2(\conj(G))$, the Fourier transform $\widehat{f}$ is indeed a function in $\ell^2(\widehat{G})$. 
Note that $\widehat{\widehat{f}}=f $ and its support is nothing but the set all conjugacy classes $C \in \conj(G)$, for them $f(C) \neq 0$. Note that for each $C\in \conj(G)$, $k_C=|C|$ which is a direct corollary of Proposition~\ref{p:Plancherel}.  Also $|\widehat{G}|=  \sum_{\pi \in \widehat{G}} d_\pi^2$ which is nothing but $\lambda(G)$.
Now applying Theorem~\ref{t:uncertainty-1} for the finite hypergroup $\widehat{G}$ we get the inequality (\ref{eq:uncertainty-compact}) again.
\end{eg}

\end{section}
\vskip1.0em

\noindent{\bf \textsf{Acknowledgements.}}
The first author was  supported by a Ph.D. Dean's Scholarship at  University of Saskatchewan and a postdoctoral fellowship at University of Waterloo for this research. The first author  thanks Y. Choi, E. Samei, and N. Spronk for their support and encouragement. The second author was partially supported by the Center of Excellence of Algebraic Hyperstructures and its Applications of Tarbiat Modares University (CEAHA), Iran.  The authors are grateful   to the referee for correcting many typos and making suggestions to improve the exposition.

\def\cprime{$'$} \def\cprime{$'$}


\normalsize
{\sc
{
\vskip0.1em
\noindent  {\bf \sc Addresses:} \newline
\indent{\bf Mahmood Alaghmandan}
\newline\indent
Department of Mathematical Sciences,
Chalmers University of Technology and   University of Gothenburg,
Gothenburg SE-412 96, Sweden}}\newline\indent
Email: \texttt{mahala@chalmers.se}

\vskip0.5em
{\sc
{\bf Massoud Amini}
\newline\indent
Department of Mathematics, Faculty of Mathematical Sciences, Tarbiat Modares University, Tehran, 14115-134, Iran
\newline\indent
}
Email: \texttt{mamini@modares.ac.ir}
\end{document}